\documentclass[12pt]{amsart}
\usepackage{color}
\usepackage[all]{xy}
\usepackage{amssymb}
\usepackage{mathrsfs}
\usepackage{eucal}

\usepackage{setspace}

\calclayout

\date{Jan.~30, 2013} 
\newtheorem{dummy}{anything}[section] 
\newtheorem{theorem}[dummy]{Theorem}
\newtheorem*{thma}{Theorem A}

\newtheorem*{Theorem}{Theorem}
\newtheorem{lemma}[dummy]{Lemma} 
\newtheorem{proposition}[dummy]{Proposition} 
\newtheorem{corollary}[dummy]{Corollary}
 
\theoremstyle{definition}

\newtheorem{Question}[dummy]{Question}

 \newtheorem{remark}[dummy]{Remark}
 \newtheorem*{Conjecture}{Conjecture} 
  
 \newtheorem*{acknowledgement}{Acknowledgement}
  \newtheorem*{Examples}{Examples}
   \newtheorem*{question}{Question}

\newcommand
{\eqncount}{\setcounter{equation}{\value{dummy}}%
\addtocounter{dummy}{1}}

\newcommand{\bZ}{\mathbf Z}

\newcommand{\bR}{\mathbb R}
\newcommand{\bbZ}{\mathbb Z}
\newcommand{\bbK}{\mathbb K}

\newcommand{\HF}{\mathbf{H}\mathfrak F}

\newcommand{\scM}{\mathscr M}

\newcommand{\cy}[1]{\bZ/{#1}}
\newcommand{\RP}{\mathbb{RP}}
\newcommand{\wX}{\widetilde X}
\newcommand{\wM}{\widetilde M}
\newcommand{\wH}{\widehat H}
\newcommand{\bd}{\partial}
\newcommand{\vv}{\, | \,}

\newcommand{\ZG}{\bZ G}
\newcommand{\ZGa}{\bZ \varGamma}
\newcommand{\BG}{B\varGamma}
\newcommand{\mmatrix}[4]{\left (\vcenter
{\xymatrix@C-2pc@R-2pc{#1&#2\\#3&#4} }
\right )}

\newcommand{\disjointunion}{\, \sqcup\,}

\DeclareMathOperator{\Sharp}{\sharp\,}

\DeclareMathOperator{\Res}{Res}

\DeclareMathOperator{\Out}{Out}
\DeclareMathOperator{\Homeo}{Homeo}

 \renewcommand{\cy}[1]{\bZ_{#1}}
 \DeclareMathOperator{\Ext}{Ext}
 
 \newcommand{\la}{\langle}
  \newcommand{\ra}{\rangle}
\newcommand{\G}{\Gamma}

\DeclareMathOperator{\vcd}{vcd}
\DeclareMathOperator{\gencd}{\underline{\textbf{cd}}}
\DeclareMathOperator{\cd}{cd}

\newcommand{\bigast}{\divideontimes}
\DeclareMathOperator{\Fix}{Fix}
\newcommand{\pts}{\text{\,pts}}

\begin{document}
\title
{More examples of discrete co-compact group actions}
\author{Ian Hambleton} 
\thanks{Partially supported by NSERC grant A4000 and the Centre for Symmetry and Deformation, University of Copenhagen}
\address{Department of Mathematics \& Statistics
 \newline\indent
McMaster University
 \newline\indent
Hamilton, ON  L8S 4K1, Canada}
\email{ian{@}math.mcmaster.ca}
\author{Erik K. Pedersen}
\address{Department of Mathematical Sciences
 \newline\indent
University of Copenhagen,
Universitetsparken 5,
 \newline\indent 2100 Copenhagen \O, Denmark}
\email{erik@math.ku.dk}

\begin{abstract} We survey  some results and questions about free actions of infinite groups on products of spheres and euclidean spaces, and give  some new co-compact  examples.
\end{abstract}
\maketitle
\section*{Introduction}
One of the main problems in geometric topology is to understand the fundamental groups of closed \emph{aspherical} smooth or topological manifolds. This problem may be expressed as follows: which discrete groups $\G \cong \pi_1(X, x_0)$ occur for some closed smooth or topological dimensional manifold $X$, with contractible universal covering  $\wX \simeq \bigast$ ? Since the homotopy type of an aspherical manifold  is completely determined by its fundamental group, the possible groups $\G$ must be finitely presented (in fact of type $FP_n$, if $\dim X = n$), and satisfy $n$-dimensional Poincar\'e duality. Although many examples are known (for example, see the work of M.~Davis \cite{davis_m_2000}), the problem is still unsolved and fascinating. 

In this paper we consider a related problem, where the universal covering space is no longer contractible:
\begin{question} Which discrete groups $\G\cong \pi_1(X, x_0)$ occur  for some closed smooth or topological manifold $X$, with  universal covering  $\wX \simeq S^{n}$ ?
\end{question}
 The assumption $\wX \simeq S^n$, implies that any finite subgroup of $\G$ must have periodic cohomology (see Swan \cite{swan1}). Equivalently, every finite subgroup of order $p^2$ must be cyclic, for each prime $p$. If $\dim X = n$, we have a formulation of the topological space form problem, and its solution by Madsen, Thomas and Wall \cite{madsen-thomas-wall1} determines precisely which finite groups $G$ can act freely on some $S^n$ by \emph{homeomorphisms}: for each prime $p$, every subgroup of $G$ with order $p^2$ or $2p$ must be cyclic. The $2p$-conditions hold if and only if $G$ contains no dihedral subgroups. Milnor \cite{milnor2} proved that the $2p$-conditions hold for a free $G$-action by homeomorphisms on $S^n$.
  
\begin{question} Which finite periodic groups can occur as subgroups of $\G = \pi_1(X, x_0)$ for some closed topological manifold $X$, with universal covering $\wX \simeq S^n$ and $\dim X > n$~?
\end{question}

For the remainder of the paper, we assume that $\dim X  > n$, so that $\G$ must be an infinite group, acting freely, properly discontinuously and co-compactly on the non-compact universal covering $\wX$. We observe that any finite periodic group satisfying all the $p^2$ and $2p$ conditions can occur as a finite subgroup in many different co-compact actions. 

\begin{Examples} Let $G$ be a finite group satisfying all the $p^2$ and $2p$ conditions, and let $\alpha\colon G \to GL_m(\bbZ)$ denote an integral representation of $G$.  Given a free action $\beta \colon G \to \Homeo(S^n)$ on $S^n$ by homeomorphisms,  the composite with the diagonal
$$(\hat\alpha \times \beta) \circ \Delta \colon G \to G \times G \to  \Homeo(T^m \times S^n)$$ gives a free $G$-action on $T^m \times  S^n$, where $T^m$ is the $m$-torus and $\hat\alpha\colon G \to \Homeo(T^m)$ is the action induced by $\alpha$.
Let $X$ be the quotient of  $T^m \times S^{n}$ by this $G$-action. The resulting fundamental groups are the semidirect products $\pi_1(X, x_0)  = \bbZ^m \rtimes_\alpha G$, 
where the conjugation action on the normal subgroup $\pi = \bbZ^m$ is induced by the representation $\alpha\colon G \to GL_m(\bbZ)$. In all these examples, the universal covering $\wX = \bR^m \times S^n$.
\end{Examples}

 Farrell and Wall (see \cite[p.~518]{wall-ppr}) asked if there was an analogue of Milnor's condition \cite{milnor2} in the setting above, which would rule out finite dihedral subgroups for co-compact actions of infinite groups ~?
In other  words, \emph{are the $2p$-conditions also necessary for finite subgroups of $\G = \pi_1(X,x_0)$ if $\dim X >n$}~?

In our earlier work \cite[Theorem 8.3]{hp1} we found the first examples of infinite discrete groups $\G$ containing finite dihedral subgroups, which could act freely, properly discontinuously and co-compactly on some product $\bR^m \times S^n$. This shows that Milnor's $2p$-conditions are \emph{not} necessary, and indicates that 
understanding exactly which finite periodic groups can appear as subgroups of $\G =\pi_1(X, x_0)$, with $\wX \simeq S^n$, is likely to be a difficult problem.

Our goal in this paper is to survey some of the previous work related to the questions above, and point out some open problems. In addition, we will discuss  \emph{virtually surface} groups $\G$, given by extensions of the form
$$1 \to \pi \to \G \to G \to 1$$
where the normal subgroup $\pi \cong \pi_1(\Sigma, x_0)$  is the fundamental group of a closed surface $\Sigma$, 
and the quotient group $G$ is finite. 
We call such an extension \emph{geometric} if the conjugation action $\alpha \colon G \to \Out(\pi)$ is induced by a group action $\rho\colon G \to \Homeo(\Sigma)$ with $\Fix(\Sigma, G)\neq\emptyset$. We assume that $\rho$ is a \emph{tame} topological action, meaning that all the fixed sets are locally-flat TOP submanifolds.
We do not assume that $\Sigma$ is orientable, or that the $G$-action on $\Sigma$ is effective.   Here is a new result about co-compact actions:

\begin{thma} Let $\pi = \pi_1(\Sigma)$, for a closed surface $\Sigma^2 \neq S^2, \RP^2$. Then the groups $\G = \pi\rtimes_\alpha D_p$ act freely, properly and co-compactly on $\bR^2 \times S^n$ if and only if $n \equiv 3 \mod 4$ and $\alpha$ is induced by an involution on $\Sigma$ with isolated fixed points.
\end{thma}

For $\Sigma = T^2$, this result is a special case of \cite[Theorem 8.3]{hp1}. Stark \cite{stark1} elaborated our existence method in \cite[\S 8]{hp1} to  provide many more examples of discrete groups $\G$ containing finite dihedral subgroups, which could act freely, properly and co-compactly on some $\bR^m \times S^n$. These examples included groups of the form $\G = \Delta\rtimes_\alpha D_p$, where $\Delta$ is the fundamental group of an orientable surface. Theorem A gives new examples for non-oriented surfaces. Our non-existence method in \cite{hp1} was extended by Anderson and Connolly \cite{anderson-connolly1} to obtain similar restrictions on finite quaternionic subgroups,

Theorem A suggests a possible general statement for co-compact actions of virtually surface groups.  Let $\G$ be a geometric extension of  $\pi = \pi_1(\Sigma, x_0)$ by finite (tame) group action $\rho\colon G \to \Homeo(\Sigma)$ with non-empty fixed set. If $H < G$ is a subgroup, then the restriction of the $G$-action on $\Sigma$ to $H$ is denoted $\Res_H(\rho)$. Suppose that all the finite subgroups of $\G$ are periodic groups, with periods dividing $2d$, and that
$\Sigma  \neq S^2, \mathbf{RP}^2$.

\begin{Conjecture} $\G$ acts freely, properly and co-compactly on $\bR^2 \times S^n$ for some $n$, if and only if $n+1 \equiv 0 \pmod{2d}$ and $\Res_H(\rho)$ is an involution on $\Sigma$ with isolated fixed points, for each finite dihedral subgroup $H \subset \G$.
\end{Conjecture}

In the next few sections we will give a short survey of related results,  and then list some open problems which we find interesting. Theorem A will be proved in Section \ref{sec:thma}.

\begin{acknowledgement}
The authors would like to thank Guido Mislin and Stratos Prassidis for helpful comments.
\end{acknowledgement}

\section{Groups with periodic cohomology}\label{sec:periodic}
The usual group cohomology theory  can be defined in two equivalent ways:
$$H^*(G; M): = \Ext^*_{\ZG}(\bZ, M) = H^*(\BG;M),$$
where $M$ is a $\ZG$-module, either by using homological algebra and derived functors, or in terms of the singular cohomology with local coefficients of the classifying space $\BG = K(G,1)$. Group cohomology can be extended to Tate cohomology $\wH^*(G;M)$, for any finite group $G$ and any $\ZG$-module $M$ (see \cite[XII]{cartan-eilenberg1} or \cite[VI]{brown1}). The key to this extension is the existence of a \emph{complete resolution} for the trivial $\ZG$-module $\bZ$, meaning an exact sequence 
$$\dots \to P_k \to \dots \to P_1 \to P_0 \to P_{-1} \to \dots \to P_{-k} \dots \to  \dots$$
of projective $\ZG$-modules $\{P_k\vv -\infty < k < \infty\}$, which agrees with a projective resolution of $\bZ$ in all positive dimensions $k \geq 0$. Tate cohomology agrees with ordinary group cohomology in positive dimensions, and with group homology in negative dimensions.

A finite group $G$ has \emph{periodic cohomology of period $q$} if there exists a class 
$\alpha \in H^q(G;\bZ)$ such that cup product
$$ \alpha\cup - :  H^i(G;M) \to H^{i+q}(G, M)$$
gives an isomorphism for all $i > 0$. As remarked above, the finite groups with periodic cohomology are exactly the ones for which every subgroup of order $p^2$ is cyclic, for every prime $p$. Alternately, these are the groups whose Sylow $p$-subgroups are either cyclic, or possibly generalized quaternion  if $p=2$. The finite groups with periodic cohomology have been classified \cite{wolf1}, and the detailed classification was an essential ingredient in the solution of the topological spherical space from problem \cite{madsen-thomas-wall1}. Periodicity for ordinary group cohomology gives periodicity for Tate cohomology as well.

For infinite groups, a useful periodicity notion was introduced by Talelli \cite{talelli1}. An infinite discrete group $\G$ has \emph{cohomology of period $q$ after $k$ steps} if  the functors $H^i(\G;-)$ and $H^{i + q}(\G;-)$ are naturally isomorphic as functors on $\ZGa$-modules, for all $i \geq k+1$. We will say that the group $\G$ is $(q,k)$-periodic if this condition holds. There is a stronger version of this idea, where there exists a class $\alpha \in H^q(\G;\bZ) = H^q(\BG;\bZ)$ such that the cup product
$$ \alpha\cup - :  H^i(\BG;M) \to H^{i+q}(\BG, M)$$
gives an isomorphism, for all $i \geq k+1$ and all $\ZGa$-modules $M$. In this case, we will say that the $CW$-complex $\BG$ \emph{has $(q,k)$-periodic cohomology} (or just that $\BG$ has periodic cohomology if the dimensions are understood). This terminology was introduced by Adem and Smith \cite[Theorem 1.2]{adem-smith1}, who showed that a connected $CW$-complex $X$ has periodic cohomology if and only if there is a spherical fibration $E \to X$ with total space homotopy equivalent to a finite dimensional $CW$-complex.

\begin{proposition}  Let $X$ be a connected $CW$-complex, and suppose that the universal covering $\wX\simeq S^n$. Then $\BG$ has periodic cohomology, where $\G = \pi_1(X,x_0)$, if and only if $X$ is homotopy equivalent to a finite dimensional $CW$-complex.
\end{proposition}

\begin{proof} This is well-known (see \cite[p.~518]{wall-ppr}). Here is a sketch of the proof, assuming that $\G$ acts trivially on the cohomology of $\wX$. Let  $\alpha \in  H^{n+1}(\BG;\bZ)$ denote the Euler class of the spherical fibration given by the classifying map $X \to \BG$ of the universal covering. Then the Gysin sequence shows that cup product with 
$$\alpha\cup - : H^i(\BG;M) \to H^{i + n+ 1}(\BG;M)$$ is an isomorphism
 for $i \geq d$ if and only if $H^j(X;M)= 0$, 
for $j \geq d+n$. By a result of Wall \cite[Theorem E]{wall-finiteness1}, the latter condition is equivalent to $X$ being finite-dimensional.
\end{proof}

\begin{corollary}
If a discrete group $\G$ acts freely and properly discontinuously  on $\bR^m \times  S^n$, then $\G$ is countable and $\BG$ has $(m,n+1)$-periodic cohomology.
\end{corollary}
\begin{proof} Since the $\G$-action is properly discontinuous, $\G$ is a  countable group.
The quotient space $X = \bR^m \times S^n/ \G$ is a topological manifold of dimension $n+m$. The periodicity dimensions follow from the Gysin sequence.
\end{proof}

 \begin{remark}\label{rem:equiv}
  A necessary condition for $\BG$ to have periodic cohomology is that every finite subgroup of $\G$ must have periodic cohomology (this is also suffficient for countable groups $\G$ if $\vcd \G < \infty$, or if $\G$ is locally finite).  
  
 It is clear that if $\BG$ has $(q,k)$-periodic cohomology, then $\G$ is $(q,k)$-periodic. The converse is conjectured to be true \cite{talelli2}, and this was verified by Talelli \cite{talelli3} for group which are periodic after 1 step, and by Mislin and Talelli \cite{mislin-talelli1} for countable groups in the Kropholler class $\HF$ of \emph{hierarchically decomposable} groups which admit a bound on the orders of finite subgroups. 

 \end{remark}

\section{Generalized cohomological dimension and Farrell cohomology}
In this section we will describe some of the algebraic background to Wall's  question 
\cite[p.~518]{wall-ppr} from 1979 about which discrete groups admit free, properly discontinuous actions on $\bR^m \times S^n$, and give some of the main results in the order they were discovered. 

A discrete group $\G$ has \emph{finite cohomological dimension} if the trivial 
$\ZGa$-module $\bZ$ has a finite length projective resolution. In this case, we write $\cd(\G) < \infty$, and we have the formula
\eqncount
\begin{equation}\label{eq:cd} \cd(\G) = \sup\{ k \vv H^k(\G; F) \neq 0, \text{\ for $F$ a free $\ZGa$-module}\}\ .
\end{equation}
If $\G$ has a finite index subgroup $\G' \leq \G$, with $\cd(\G') < \infty$, then we say that $\G$ has \emph{finite virtual cohomological dimension} and write $\vcd\G  < \infty$. 

 Johnson (1985,\cite{johnson3}) answered Wall's question in a special case, and constructed many \emph{geometric} examples of free, proper (and even co-compact) actions for groups arising from uniform  lattices in non-compact Lie groups.  

 \begin{theorem}[{\cite[Corollary 2]{johnson3}}] Let $\G$ be an extension
 $$1 \to A \to \G \to G\to 1$$
where $A$ is a normal subgroup of $\G$ such that $K(A,1)$ is a finite complex, and $G$ is a finite periodic group. Then $\G$ acts freely and properly on some $\bR^m \times S^n$.
\end{theorem}

For discrete groups with $\vcd\G <\infty$, Farrell \cite{farrell2} introduced and studied a generalization of Tate cohomology (and Tate homology), now usually called \emph{Farrell cohomology}, and denoted $\wH^*(\G;M)$. There is natural map
$$\iota \colon H^*(\G;M) \to   \wH^*(\G;M)$$
from the usual group cohomology, which is an isomorphism in large positive dimensions. The main step in Farrell's construction was again to produce a complete resolution $\{P_k\}$, $-\infty < k < \infty$,  of projective $\ZGa$-modules, but this time generalized to require only that the resolution should agree with a projective resolution of $\bZ$ in sufficiently high dimensions $k \geq N$ (depending on the virtual dimension $\vcd \G$).
For finite groups, Farrell cohomology reduces to Tate cohomology. 

Brown \cite[Chap.~X]{brown1} proved that the Farrell cohomology of $\G$ has no elements of infinite order, and is determined by the lattice of finite subgroups of $\G$. In particular, it follows that a discrete group $\G$ with $\vcd \G<\infty$ has periodic Farrell cohomology (in the same sense as for ordinary cohomology) if and only if every finite subgroup $H \leq \G$ has periodic cohomology.

\smallskip
A breakthrough on Wall's question was made by Connolly and Prassidis (1989):

\begin{theorem}[{\cite{connolly-prassidis1}}] Suppose that $\vcd\G < \infty$. The $\G$ acts freely and properly on some $\bR^m\times S^n$ if and only if $\G$ is countable and has periodic Farrell cohomology.
\end{theorem}

Ikenaga \cite{ikenaga1,ikenaga2} introduced a \emph{generalized cohomological dimension}, denoted $\gencd \G$, and an extension of Farrell cohomology.  For any discrete group $\G$, define
 $$\gencd\G = \sup\{k \vv \Ext^k_{\ZGa}(M, F) \neq 0, \text{\ for $M$ a $\bZ$-free $\ZGa$-module, and $F$ free over $\ZGa$}\}.$$
 Since $H^k(\G;F) = \Ext^k_{\ZGa}(\bZ, F)$, this is a natural generalization (replacing $\bZ$ in \eqref{eq:cd} by an arbitrary $\bZ$-free module $M$).
Here are some properties of Ikeanaga's generalized dimension.
\begin{enumerate}
\item $\gencd \G \leq \cd \G$, with equality if $\cd \G < \infty$. 
\item For any subgroup $\G' \leq \G$, $\gencd \G' \leq \gencd \G$, with equality if the index $[\G:\G'] < \infty$.
\item If $\vcd\G < \infty$, then $\gencd\G = \vcd\G$.
\item For $G$ finite, $\gencd G = 0$.
\item If $\gencd \G = 0$, then $\G$ is a torsion group.
\end{enumerate}

 Ikenaga \cite[p.~433]{ikenaga1} also defined and investigated a generalization of Farrell cohomology $\wH^*(\G;M)$ for those discrete groups $\G$ with $\gencd\G <\infty$, which admit a complete resolution of projective $\ZGa$-modules (agreeing with a projective resolution of $\bZ$ in sufficiently large positive dimensions. If $\vcd\G < \infty$, then the generalized Farrell cohomology agrees with the original version. One new feature is that the generalized Farrell cohomology can contain elements of infinite order. Ikenaga \cite[\S 5]{ikenaga1} provided an interesting class $\mathcal C_\infty$ of discrete groups with $\gencd <\infty$ which admit complete resolutions, so the generalized Farrell cohomology is defined. Mislin \cite{mislin2} later extended Farrell cohomology to \emph{arbitrary} discrete groups, such that the extended version agrees with Ikenaga's generalized Farrell cohomology whenever it is defined (see the discussion in \cite[\S 2]{mislin-talelli1}). 

\smallskip
The next step towards Wall's question was by taken by Prassidis (1992):
\begin{theorem}[{\cite[Theorem 10]{prassidis1}}] There exist discrete groups $\G$ with $\vcd\G = \infty$ which act freely and properly on some $\bR^m\times S^n$.
\end{theorem}
 The examples found by Prassidis are for certain groups $\G$ in Ikenaga's class ${\mathcal C}_\infty$, which have $\gencd\G <\infty$, but $\vcd\G = \infty$, and periodic generalized Farrell cohomology.

\begin{remark} If $\G$  acts freely and properly on some $\bR^m \times S^n$, then $\gencd \G < \infty$. 
This follows by combining two results of Mislin\footnote{We are indebted to Guido Mislin for this remark.} and Talelli \cite[Theorem 2.5, Corollary 5.2]{mislin-talelli1}.
\end{remark}

 The relation between  periodicity for generalized Farrell cohomology and Talelli's periodicity was explained by Prassidis \cite[\S 2]{prassidis1}.
\begin{proposition}[{\cite[Prop.~3]{prassidis1}}] Let $\G$ be a group with $\gencd\G < \infty$. Then the following are equivalent:
\begin{enumerate}
\item The group $\G$ admits a complete resolution, $\gencd\G \leq k$, and $\G$ has periodic generalized Farrell cohomology of period $q$.
\item $\G$ has $(q,k)$-periodic cohomology.
\end{enumerate}
\end{proposition}
Finally, Adem and Smith (2001) completely answered Wall's original question:
\begin{theorem}[{\cite[Corollary~1.3]{adem-smith1}}] A discrete group $\G$ acts freely and properly on some $\bR^m \times S^n$ if and only if $\G$ is countable and $\BG$ has periodic cohomology.
\end{theorem}

It would be interesting to know if there are groups which have $(q,k)$-periodic cohomology, but do not act freely and properly on any $\bR^m \times S^n$. After the solution of Wall's problem, this is  now primarily a question in group cohomology (see Remark \ref{rem:equiv}).

\section{Co-compact actions on $\bR^m \times S^n$}

We now return to the setting described in the Introduction, where $X$ is a closed topological manifold with universal covering $\wX \simeq S^n$. We are interested in determining the possible fundamental groups $\G = \pi_1(X, x_0)$ for such manifolds.   In this section, we assume that $\wX = \bR^m \times S^n$ for some $n$, $m$.
\begin{Question} Which discrete group $\G$ can act freely, properly discontinuously on $\bR^m \times S^n$ with compact quotient~?
\end{Question}

From the solution of Wall's problem by Adem and Smith \cite{adem-smith1}, we know that any $\G$ which admits a free co-compact action must be countable and $\BG$ must have periodic cohomology. In particular, every finite subgroup of $\G$ must have periodic cohomology (see Kulkarni \cite{kulkarni3} for some nice co-compact examples arising from differential geometry).

Farrell and Wall \cite[p.~518]{wall-ppr} asked specifically for some analogue of Milnor's condition:

\begin{Question} If $\G$ acts freely, properly and co-compactly on $\bR^m \times S^n$, then could $\G$ contain a finite dihedral subgroup~?
\end{Question}
 Consider the groups $\G =\bbZ^k \rtimes_\alpha D_p$, obtained as the semi-direct product of a finite dihedral group $D_p$, for $p$ an odd prime, and a free abelian group $\bbZ^k$, via an integral representation $\alpha\colon D_p \to GL_k(\bbZ)$. It turns out that some of these groups can act co-compactly on $\bR^m \times S^n$, so finite dihedral subgroups are not ruled out.

\begin{Theorem}[{\cite[Theorem 8.3]{hp1}}] The group $\G = \bbZ^k \rtimes_\alpha D_p$ acts freely, properly, and co-compactly on $\bR^m \times S^n$, if and only if $ n \equiv 3 \pmod 4$, $m = k$, and $\alpha$ considered as a real represenation has at least two $\bR_-$ summands.
\end{Theorem}
In the statement $\bR_-$ denotes the non-trivial 1-dimensional real representation of $D_p$. 
This result greatly complicates the finite subgroups problem for co-compact group $\G$-actions on $\bR^m \times S^n$. What other sorts of conditions are possible ?
\begin{Question} If $\G$ acts freely, properly and co-compactly on $\bR^m \times S^n$, then is  there a bound on the order of finite subgroups ? Could $\G$ contain infinitely many conjugacy classes of finite subgroups ? 
\end{Question}
Work of Leary and Nucinkis \cite{leary-nucinkis1, leary4} based on results of Bestvina, Brady, Bridson and Oliver shows that there are finitely-presented groups of type $VF$ (i.e~containing a normal finite index subgroup $\G_0\leq \G$ with $B\G_0$ finite) which do not satisfy the conjugacy finiteness condition on subgroups. 
\begin{Question} If $\G$ acts freely, properly and co-compactly on $\bR^m \times S^n$, and $\vcd\G < \infty$, then is $\G$ a virtual Poinca\'e duality group ?
\end{Question}
The closest approach to this problem so far is another result of Connolly and Prassidis: 
\begin{theorem}[{\cite[Theorem C]{connolly-prassidis1}}] Suppose that $\vcd\G <\infty$, and that $\G$ has only finitely many conjugacy classes of finite subgroups. If $N_\G(H)$ has type $F(\infty)$ for each finite subgroup $H \leq \G$, then the following are equivalent:
\begin{enumerate}
\item $\wH^*(\G;\bZ)$ is periodic and some finite index subgroup of $\G$ is a Poincar\'e duality group.
\item There is a finite Poincar\'e complex $X$ with $\pi_1(X, x_0) = \G$ and $\wX \simeq S^n$.
\end{enumerate}
\end{theorem}

Finally, 
Wall \cite{wall-ppr} also asked: if a discrete group $\G$ acts freely, properly discontinuously and co-compactly on $\bR^m \times S^n$, for some $m, n$, then is $\vcd \G < \infty$ ?  This was answered by Farrell and Stark \cite{farrell-stark1}, who found examples with infinite virtual cohomological dimension, based on work of Raghunathan \cite{raghunathan1}. The groups used in these examples are subgroups of Lie groups, but are not residually finite. Each one contains a finite central subgroup $A < \G$, such that $\vcd (\G/A) < \infty$. 
In addition, we note that $\gencd\G < \infty$ for these groups by \cite[Theorem 3.2]{ikenaga2}. 
The authors wondered \cite[p.~4]{farrell-stark1}:  ``whether every co-compact spherical-Euclidean group can be obtained from a group of finite cohomological dimension by combinations of central extensions with finite kernels and extensions with finite quotients ?" This suggests the following:
\begin{Question}
if a residually finite discrete group $\G$ acts freely, properly discontinuously and co-compactly on $\bR^m \times S^n$, for some $m, n$, then is $\vcd \G < \infty$ ?
\end{Question}
We conclude by mentioning another more speculative source of restrictions on groups admitting co-compact actions (see \cite{stark2}, \cite{hp3} for some discussion).
\begin{Question} What can be said about Gromov's  \emph{large-scale} geometry of discrete groups  which acts freely, properly and co-compactly on $\bR^m \times S^n$ ?
\end{Question}

\section{Involutions on the Klein bottle}\label{sec:thma}
It is well-known that there are three distinct involutions (up to conjugacy) on the torus $T^2$ with quotient a closed surface. If we represent points on $T^2 = S^1 \times S^1$ as pairs $(z, w)$ of unit length complex numbers, then we have:
\begin{enumerate}
\item $\sigma(z, w) = (\bar z, \bar w)$
\item $\tau(z, w) = (-z, \bar w)$
\item $\lambda(z, w) = (-z, w)$
\end{enumerate}
In the first case, $\Fix(T^2, \sigma)$ consists of 4 isolated points, and the quotient is $S^2$; in the second case,  $\Fix(T^2, \tau) = \emptyset$ and the quotient is the Klein bottle $\bbK^2$. We have a diagram of compatible quotients:
$$\xymatrix{T^2 \ar[r]^{/\sigma}\ar[d]_{/\tau} & S^2\ar[d]\cr
 \bbK^2 \ar[r]^{/\alpha} & \RP^2}$$
 The projection $S^2 \to \RP^2$ is the usual double covering; the remaining involution $\alpha\colon \bbK^2 \to \bbK^2$ has 2 isolated fixed points and the quotient space is $\RP^2$.
 
  Let $\scM = \RP^2 \setminus int\, D^2$ denote the Moebius band, with
   $\bd \scM = S^1$. We observe that
$\bbK^2 = \scM \cup_r \scM$
where $r\colon \bd \scM = S^1 \to S^1 = \bd\scM$ is given by $r(z) = \bar z$. For example, this follows because the surface obtained by the glueing is non-orientable and has Euler characteristic zero. 
Here is a geometric description of $(\bbK^2, \alpha)$: 
 
 \begin{lemma} The involution $(\bbK^2, \alpha)$ extends the reflection $r\colon \bd \scM \to \bd\scM$ by interchanging the two copies of $\scM$.
\end{lemma}
From this description, we see that $(\bbK^2, \alpha)$ has two isolated fixed points and quotient $\RP^2$. Now we can define an involution on any non-orientable surface
$$(\Sigma^2, \alpha)  = (\bbK^2, \alpha)\Sharp (\bbK^2, \alpha) \Sharp \dots \Sharp (\bbK^2, \alpha)$$
by taking the equivariant connected sum at  the fixed points in successive copies of $(\bbK^2, \alpha)$. The resulting involution $(\Sigma^2, \alpha)$ has 2 isolated fixed points.

We will also need some information about the fundamental groups in the ramified 
covering $\bbK^2 \to \RP^2$. In the quotient, the van Kampen diagram
$$\xymatrix{& \pi_1(D^2 - 2 \pts) = F(a,b)\ar[dr]& \cr
\pi_1(S^1) \ar[ur]\ar[dr]&& G = \la a,b,c\vv c^2 = ab^{-1}\ra\cr
& \pi_1(\scM) = \bbZ \la c\ra\ar[ur]&}$$
yields the isomorphism
$G \cong \pi_1(\RP^2 - 2 \pts) = F(a,c)$, with the free group on two generators $a$, given by a loop around one of the fixed points,  and $c$ given by the core of the Moebius band $\scM$. A loop around the other fixed point represents $b = ac^{-2}$.  The regular covering  over $(\RP^2 - 2 \pts)$ is induced by the homomorphism $\varphi\colon F(a,c) \to \cy 2 \cong \{\pm 1\}$ defined by $\varphi(c) = 1$, and $\varphi(a) = \varphi(b) = -1$. 
Since the loop $\bd D^2$ represents $ab^{-1} = c^2$, we see that the covering $\bbK^2 \to \RP^2$ is unramified  on the complement of two $\cy 2$-invariant disks $H_2 \disjointunion H'_2$ (or equivariant $2$-handles) around the fixed points (in $\bbK^2$) union two annuli (or thickened 1-handles) $H_1 \disjointunion \alpha H_1$ joining the 2-handles.

\begin{proof}[The proof of Theorem A] 
The basic step for the existence part of our earlier result \cite[Theorem 8.3]{hp1} about dihedral subgroups was the solution of a compact  ``blocked surgery problem":
$$ (f \times 1, b\times 1)\colon \wM \times_{D_p} T^2 \to \wX \times_{D_p} T^2$$
where $(f,b)\colon M \to X$ was a degree 1 normal map from a closed $n$-manifold $M$ to a finite Swan complex $X$ for the dihedral group $D_p$. 
We formed the $D_p$-balanced quotient by using the diagonal $-1$ action of $D_p$ on $T^2$, via the projection $D_p \to \cy 2 \to GL_2(\bbZ)$.

Recently we looked again at our blocked surgery argument, and noticed that the role of the torus $T^2$ could be replaced by any non-orientable surface (Stark \cite{stark1} used orientable surfaces of any genus $g \geq 1$).
We consider the blocked surgery problem $$(f\times 1, b\times 1)\colon \wM^n \times_{D_p} \Sigma^2 \to \wX \times_{D_p} \Sigma^2,$$ where $X$ is an $n$-dimensional finite Swan complex for $D_p$,  which acts on the non-orientable closed surface $\Sigma^2$ by the involution $\alpha$ constructed above. 
The existence part of Theorem A now follows exactly the same steps as in the proof of our previous blocked surgery result \cite[Theorem 8.2]{hp1}. We use the 
equivariant handle information for $(\Sigma, \alpha)$ as a ramified covering to specify the blocks.
 
In the other direction, the condition given in Theorem A is also necessary. Assume that the group $\G = \pi\rtimes_\alpha D_p$ act freely, properly and co-compactly on $\bR^2 \times S^n$, where $\pi = \pi_1(\Sigma)$, for a closed surface $\Sigma^2 \neq S^2, \RP^2$. Then by \cite[Theorem A]{hp3}, the $\G$-action can be compactified to an action on $S^{n+2}$, which is free away from  an invariant $S^1 \subset S^{n+2}$ where the action is
 induced by $\alpha$. Now the necessity of our condition on $\alpha$ follows from \cite[Theorem 7.11]{hp3}.
 \end{proof}

\providecommand{\bysame}{\leavevmode\hbox to3em{\hrulefill}\thinspace}
\providecommand{\MR}{\relax\ifhmode\unskip\space\fi MR }
\providecommand{\MRhref}[2]{%
  \href{http://www.ams.org/mathscinet-getitem?mr=#1}{#2}
}
\providecommand{\href}[2]{#2}

\end{document}